\documentclass{article}

\usepackage{amsthm,amsmath,amsfonts}
\usepackage{enumerate}
\usepackage[abbrev]{amsrefs}

\newcommand{\C}{\mathbb{C}}
\newcommand{\R}{\mathbb{R}}
\newcommand{\Z}{\mathbb{Z}}
\newcommand{\Q}{\mathbb{Q}}

\usepackage{threeparttable}
\usepackage{graphics}
\usepackage{caption}
\usepackage{lipsum}
\usepackage{longtable}
\usepackage{tabu}
\usepackage{threeparttablex}

\usepackage{mathtools}

\usepackage{float}
\floatstyle{plaintop}
\restylefloat{table}

\newtheorem{theorem}{Theorem}
\newtheorem{proposition}{Proposition}
\theoremstyle{definition}
\newtheorem{definition}{Definition}
\newtheorem{example}{Example}

\newtheorem{lemma}{Lemma}

\theoremstyle{remark}
\newtheorem*{remark}{Remark}

\newcommand\N{\mathbb N}

\usepackage{float}
\restylefloat{table}

\numberwithin{equation}{section}
\numberwithin{theorem}{section}
\numberwithin{lemma}{section}
\numberwithin{proposition}{section}
\numberwithin{definition}{section}
\numberwithin{example}{section}

\title{On the co-factors of degree 6 Salem number beta expansions}
\date{\today}
\author{Jacob J. Stockton\thanks{Research partially supported by NSERC grant 2019-03930 and the University of Waterloo President's Research Award.}\\Department of Pure Mathematics\\University of Waterloo\\Waterloo, ON, Canada, N2L 3G1\\ \texttt{jacob.stockton@hotmail.com}}

\begin{document}
\maketitle

\begin{abstract}
For $\beta > 1$, a sequence $(c_n)_{n \geq 1} \in \Z^{\N^+}$ with $0 \leq c_n < \beta$ is the \emph{beta expansion} of $x$ with respect to $\beta$ if $x = \sum_{n = 1}^\infty c_n\beta^{-n}$. Defining $d_\beta(x)$ to be the greedy beta expansion of $x$ with respect to $\beta$, it is known that $d_\beta(1)$ is eventually periodic as long as $\beta$ is a Pisot number. 
It is conjectured that the same is true for Salem numbers, but is only currently known to be true for Salem numbers of degree 4.
Heuristic arguments suggest that almost all degree 6 Salem numbers admit periodic expansions but that a positive proportion of degree 8 Salem numbers do not.
In this paper, we investigate the degree 6 case.
We present computational methods for searching for families of degree 6 numbers with eventually periodic greedy expansions by studying the co-factors of their expansions.
We also prove that the greedy expansions of degree 6 Salem numbers can have arbitrarily large periods.
In addition, computational evidence is compiled on the set of degree 6 Salem numbers with $\text{trace}(\beta) \leq 15$. We give examples of numbers with $\text{trace}(\beta) \leq 15$ whose expansions have period and preperiod lengths exceeding $10^{10}$, yet are still eventually periodic.
\end{abstract}


\section{Introduction}\label{sec:introduction}

Beta expansions were first introduced by R\'enyi \cite{renyi} in 1957 as a generalization of the base-$b$ representations of real numbers to non-integer bases. Suppose that $\beta > 1$ is a real number and $(c_n)_{n \geq 1}$ is an integer sequence with $0 \leq c_n < \beta$ such that
\[ x = \sum_{n = 1}^\infty c_n\beta^{-n}, \]
for some $x$. Then the sequence $(c_n)_{n \geq 1}$ is said to be the \emph{beta expansion} of $x$ with respect to $\beta$. 
Note that beta expansions for a fixed $x$ are not necessarily unique; in fact, if $\beta = (1+\sqrt{5})/2$, then $x = 1$ has a (countably) infinite number of distinct beta expansions. We therefore define the canonical ``greedy'' beta expansion of $x$ to be the largest beta expansion of $x$, ordered lexicographically. The greedy expansion of a number $x$ may be calculated by using a well-known algorithm due to R\'enyi: begin with $r_0 := x$, and for $n \geq 1$, set $c_n := \lfloor \beta r_{n-1} \rfloor$ and $r_n := \beta r_{n-1} - \lfloor \beta r_{n-1} \rfloor$. Then $(c_n)_{n \geq 1}$ is the greedy beta expansion of $x$, and we write $d_\beta(x) := (c_n)_{n \geq 1}$. 
In this paper, all beta expansions are assumed to be greedy.

One of the first questions that was asked about beta expansions was whether or not the sequence of digits $d_\beta(x)$ in the greedy expansion is eventually periodic.
More precisely, if there exists $p \geq 1$ such that $c_{n} = c_{n+p}$ for all $n \geq 1$, then the sequence $(c_n)_{n \geq 1}$ is said to be \emph{periodic}.
If $(c_{n+m})_{n \geq 1}$ is periodic for some $m \geq 0$, then $(c_n)_{n \geq 1}$ is \emph{eventually periodic}.
For $\beta > 1$, R\'enyi \cite{renyi} and Parry \cite{parry} studied the eventual periodicity of $d_\beta(1)$ by considering the so-called beta transformation $T_\beta \colon [0, 1] \to [0, 1)$ defined by 
\[ T_\beta(x) := \beta x \text{ mod } 1. \]
Numbers for which the orbit $\{T_\beta^n(1)\}_{n \geq 1}$ is finite are called \emph{Parry numbers} (they were called \emph{beta numbers} by Parry). In the special case where $T_\beta^N(1) = 0$ for some $N$,  $\beta$ is called a \emph{simple} Parry number.
Observe that in the algorithm for calculating greedy beta expansions, one has $r_n = T_\beta^n(1)$ and hence $\beta$ is a Parry number if and only if $d_\beta(1)$ is eventually periodic. In this case, the smallest $m \geq 0$ and $p \geq 1$ such that $c_{n+m} = c_{n+m+p}$ for all $n \geq 1$ are called the \emph{preperiod length} and \emph{period length} of the beta expansion, respectively.
Note that for simple Parry numbers, we have $c_n = 0$ for all $n \geq m$ and so (by ignoring trailing zeros) these expansions can be considered to be finite. We therefore define $p = 0$ for simple Parry numbers and write $d_\beta(1) = c_1, \dots, c_m$. When $p > 0$, we instead write $d_\beta(1) = c_1, \dots, c_m : c_{m+1}, \dots, c_{m+p}$.

Recall that a \emph{Pisot number} (or sometimes \emph{Pisot-Vijayaraghavan number}) is an algebraic integer $\beta > 1$ whose Galois conjugates $\alpha$ all satisfy $|\alpha| < 1$. A \emph{Salem number} is an algebraic integer $\beta > 1$ whose Galois conjugates $\alpha$ all satisfy $|\alpha| \leq 1$ with at least one conjugate satisfying $|\alpha| = 1$. If $\beta$ is Pisot or Salem then its minimal polynomial $P(x)$ may simply be referred to as a Pisot or Salem polynomial, respectively. 

It was shown by Schmidt \cite{schmidt} that if all $x \in \Q \cap [0, 1]$ have periodic $\beta$-expansions, then $\beta$ is either a Pisot or a Salem number. Schmidt proved a partial converse, which was independently proven by Bertrand \cite{bertrand}, showing that if $\beta$ is Pisot then $\beta$ is a Parry number. Schmidt conjectured that the same was true for Salem numbers. In \cite{boyd-s4}, Boyd answers this conjecture positively for Salem numbers of degree 4; however, it is still not known for any degree greater than 4. 
The degree 6 case already does not appear to be as well behaved as the degree 4 case. In particular, one always has $m+p < \text{trace}(\beta)$ for degree 4 Salem numbers, but no similar bound seems to hold for degree 6 Salem numbers.
Boyd showed that the Salem number with minimal polynomial
\begin{align}\label{eq:largeexp}
P(x) = x^6 - 3x^5 - x^4 - 7x^3 - x^2 - 3x + 1
\end{align}
has $m+p > 1.1 \times 10^9$, presenting it as a possible counter-example to Schmidt's conjecture. As part of this paper, we further this bound to $m+p > 7.7 \times 10^{10}$.

It is straightforward to see that if $\beta$ is a simple Parry number, i.e. with finite expansion $c_1, \dots, c_m$, then $\beta$ is an algebraic integer satisfying $P_m(\beta) = 0$, where $P_m(x)$ is defined by
\begin{equation}\label{eq:P_m(x)}
P_m(x) := x^m - c_1x^{m-1} - \dots - c_m.
\end{equation}
Furthermore, it can be shown that if $\beta$ is any Parry number, then $\beta$ is a root of the polynomial
\begin{align}\label{eq:r(x)}
 R(x) = \begin{cases} P_{m+p} - P_m(x) &p > 0 \\ P_m(x) &p=0   \end{cases}.
 \end{align}
In particular, if $\beta$ is a Salem number with minimal polynomial $P(x)$, we may write $R(x) = P(x)Q(x)$ for some integer polynomial $Q(x)$. The polynomial $R(x)$ is called the \emph{companion polynomial} of $\beta$ and $Q(x)$ is the \emph{co-factor} of $\beta$. It was shown by Parry \cite{parry} that the roots of $R(x)$ other than $\beta$ must lie in the disk $|z| < \text{min}(2, \beta)$. This bound was improved to $|z| \leq (1 + \sqrt{5})/2$ by Flatto, Lagarias, and Poonen \cite{flatto} and by Solomyak \cite{solomyak}, independently. Hence the roots of the co-factor $Q(x)$ must, also, all lie in the disk $|z| \leq (1 + \sqrt{5})/2$. 

A few elementary facts about the expansions of Salem numbers become apparent after this discussion. First, if $\beta$ is a degree $n$ Salem Parry number, then its expansion must have $m+p \geq n$. Second, a simple result from \cite{boyd-s4} is Salem numbers cannot be simple Parry numbers for if there was such a $\beta$, then $\beta$ would be a root of \eqref{eq:P_m(x)}. However, \eqref{eq:P_m(x)} has at most one positive root, and the minimal polynomial of $\beta$ has the two positive roots $\beta$ and $1/\beta$.

The purpose of this paper is to study the beta expansions of degree 6 Salem numbers in view of Schmidt's conjecture. This is primarily done through the study of the co-factors of these expansions. In Section \ref{sec:beta expansions of salem numbers} we provide explicit descriptions of which Salem numbers of degree 6 have expansions with $(m, p) = (1, 7)$, and we give the co-factors of these expansions. 
The main difference encountered when $p > 6$ (as opposed to when $p \in \{5, 6\}$, as studied in \cite{boyd-s6}) is that there can be more that one possible co-factor for a given $(m, p)$.
In Section \ref{sec:computing co-factors}, we detail computational methods for determining all co-factors for a given expansion length $(m, p)$. 
We also compute all possible co-factors for expansions with $(m, p) = (1, p)$ where $5 \leq p \leq 10$.

One other important difference between the degree 4 case and the degree 6 case which we investigate here is that the co-factors of degree 6 Salem numbers are not always cyclotomic. Section \ref{sec:cyclotomic co-factors} gives a simple criterion in terms of the coefficients of the minimal polynomial to determine when no cyclotomic co-factors can exist.

In Section \ref{sec:expansions for large trace}, we show that degree 6 Salem numbers can have expansions with arbitrarily large periods. We do this by exhibiting an infinite family of Salem numbers, parameterized by $k \in \N$, which have $(m, p) = (1, 8k+6)$.

\section{Beta expansions of Salem numbers}\label{sec:beta expansions of salem numbers}

\subsection{Salem numbers of degree 6}

We turn our attention to Salem numbers of degree 6. A result from \cite{salem}*{p.26} is that all Salem polynomials $P(x)$ satisfy
\[ x^{\text{deg}P(x)}P(x^{-1}) = P(x), \]
i.e. $P(x)$ is a reciprocal polynomial.
This can be seen \cite{smyth}*{Lemma 1} by noticing that if $P(x)$ is a Salem polynomial with root $\beta > 1$, then it has a root $\alpha \neq \pm 1$ on the unit circle; hence, a Galois automorphism mapping $\alpha \mapsto \beta$ will send $\alpha^{-1} \mapsto \beta^{-1}$. Therefore the roots of $P(x)$ appear in reciprocal pairs, and all conjugates of $\beta$ other than $\beta^{-1}$ lie on the unit circle. An immediate consequence is that $\deg P(x)$ is even and $\deg P(x) \geq 4$. 

Degree 6 Salem polynomials have the form
\begin{equation}\label{eq:deg6salem}
P(x) = x^6 + ax^5 + bx^4 + cx^3 + bx^2 + ax + 1,
\end{equation}
which we often abbreviate by simply writing $P(x) = (a, b, c)$. 
Boyd \cite{boyd-s6} gave explicit descriptions for which Salem numbers of degree 6 have $(m, p) = (1, 5)$ and $(m, p) = (1, 6)$, and he gave a sufficient condition for having $(m, p) = (5, 33)$. His conditions are presented in terms linear inequalities involving the coefficients of the minimal polynomial of $\beta$. He also presented a probabilistic heuristic argument, based on the assumption that the points in the orbit $\{T_\beta^n(1)\}_{n \geq 1}$ are distributed randomly. 
This model predicts that the expected value of $m+p$ is roughly given by
\begin{align*}
E(m+p)\begin{cases} < \infty &\text{if $C(\beta) < 1$}, \\ = \infty &\text{if $C(\beta) \geq 1$}. \end{cases}
\end{align*}
where $C(\beta) = (\pi/6)^2 \beta^5 / \text{disc}(\beta)^{1/2}$.
The same model predicts that almost all degree 6 Salem numbers are Parry numbers, but that a positive proportion of degree 8 Salem numbers are not Parry numbers.
Recently, Hichri \cite{hichri} offered some computational evidence supporting this claim for degree 8 Salem numbers. He also generalized some of Boyd's results to degree 8 Salem numbers, giving a complete description of degree 8 Salem numbers having $(m, p) = (1, 7)$ and $(m, p) = (1,8)$. 

The following result is useful in recognizing degree 6 Salem numbers.
\begin{lemma}[\cite{boyd-s6}*{Lemma 2.1}]\label{ureqs}
Let $P(x) = (a,b, c)$ be as in \eqref{eq:deg6salem}. If the polynomial
\[ U(x) = x^3 + ax^2 + (b-3)x + (c-2a) \]
satisfies $U(\pm 2) < 0$ and $U(n) \neq 0$ for $-1 \leq n \leq 1 + \max(|a|, |b-3|, |c-2a|)$, as well as any of one of $U(-1) > 0$, $U(0) > 0$, or $U(1) > 0$, then $P(x)$ is a Salem polynomial.
\end{lemma}

Boyd \cite{boyd-s6} uses this criterion to construct a complete table of Salem numbers of degree 6 and trace at most 15. There are 11836 such numbers, and he was successful in finding eventually periodic expansions for all but 81 of these. For these 81 exceptions, including the Salem root of \eqref{eq:largeexp}, he showed that the expansion would have to be exceptionally large. Table \ref{tab:largeexps} shows all of the Salem numbers with $\text{trace}(\beta) \leq 10$ whose beta expansions were previously unknown. In particular, we have found eventually periodic expansions for 6 out of 18 these numbers. We have also improved the bound on $m+p$ for many others, including showing that $m+p > 7.7 \times 10^{10}$ for \eqref{eq:largeexp}. Table \ref{tab:all data} at the end of this paper gives similar data on all 81 previously unknown expansions. 

Previously, it was thought that examples such as \eqref{eq:largeexp} could be potential counterexamples to Schmidt's conjecture. 
In view of the very large expansions discovered in Table \ref{tab:largeexps} and Table \ref{tab:all data}, it may simply be the case that these examples have arbitrary large (yet eventually periodic) expansions relative to $\text{trace}(\beta)$.
Nevertheless, Boyd's heuristic argument does not reject the possibility that there may exist certain degree 6 Salem numbers for which $m+p = \infty$.

\begin{table}
\centering
\caption{Large beta expansions for $\text{trace}(\beta) \leq 10$} \label{tab:largeexps}
\begin{threeparttable}
\begin{tabular}{ r|c|c|c }
 $(a, b, c)$ & $m$ & $p$ & $m+p >$ \\
 \hline
 $(-3, -1, -7)$     & $*$ & $*$ & $77930975072$ \\  
 $(-5, -2, -11)$    & $26440386599$ & $6051049471$ & N/A \\
 $(-6, -26, -39)$   & $*$ & $*$   &     $86487966351$  \\
 $(-6, -21, -31)$   & $*$ & $*$   &     $84721475914$  \\
 $(-7, -29, -43)$   & $1039779$   &  $90$ &  N/A \\
 $(-7, -28, -41)$   & $159781799$ & $94829$ & N/A  \\
 $(-8, -33, -49)$   & $*$ & $*$   & $29951657970$  \\
 $(-8, -30, -44)$   & $2121493281$ & $188611456$ & N/A \\
 $(-8, -26, -38)$   & $*$ & $*$   &   $90758317726$ \\
 $(-8, -23, -34)$   & $*$ & $*$   &      $82585386156$  \\
 $(-8, -3, -17)$    & $*$ & $*$   &   $33565256553$  \\
 $(-9, -35, -51)$   & $4075324464$ & $425719617$ & N/A\\
 $(-9, -28, -41)$   & $*$ & $*$ & $77441722314$ \\
 $(-9, -6, -20)$    & $*$ & $*$  &  $41820094414$    \\
 $(-10, -41, -61)$  & $*$ & $*$ &   $58743875586$    \\
 $(-10, -40, -59)$  & $*$ & $*$  &  $76519283803$  \\
 $(-10, -36, -52)$  & $23087045143$ & $820400$ &  N/A \\
 $(-10, -4, -21)$   & $*$ & $*$ &    $33346343238$
\end{tabular}
\begin{tablenotes}\footnotesize
\item[*] The beta expansion is still unknown for 12 values of $(a, b, c)$. For these, a lower bound on $m+p$ is given.
\end{tablenotes}
\end{threeparttable}
\end{table}

\subsection{Salem numbers with $(m, p) = (1, p)$}\label{sec:salem numbers with (m,p)=(1,p)}

In \cite{boyd-s4}, it is shown that all degree 4 Salem numbers have periodic expansions with $m = 1$. While this is not true for higher degrees, of all degree 6 Salem numbers with trace at most 15, about 81\% have $m = 1$.
Hence, it may be worth investigating this case further.
Recall that when $m = 1$, the companion polynomial for $\beta$ is given by
\begin{align}\label{eq:rxm1}
R(x) &= P_{p+1}(x) - P_1(x) \nonumber \\
&= x^{p+1} - c_1x^{p} - \dots - c_{p-1}x^2 - (c_p + 1)x + (c_1 - c_{p+1}).
\end{align}
In particular, if we know the coefficients of $R(x)$ by some other method, then we can easily deduce the beta expansion of $\beta$ by simply reading the coefficients of $R(x)$. Conversely, if we know the beta expansion, then $R(x)$ is relatively simple to understand.

The following proposition gives a simple result regarding the co-factor of this expansion.
\begin{proposition}
If $Q(x)$ is the co-factor of a Salem number with expansion $(m, p) = (1, p)$, then $Q(x)$ has no positive roots.
\end{proposition}
\begin{proof}
Say $\beta$ is a Salem number with minimal polynomial $P(x)$ and co-factor $Q(x)$. Suppose its beta expansion is given by
\[ c_1 : c_{2}, \dots, c_{p+1}. \]
Then its companion polynomial is given by \eqref{eq:rxm1}.
Note that coefficients of $R(x)$ are all non-positive except for the leading and trailing coefficients. By Descartes' rule of signs, say, $R(x)$ has at most two positive roots. However, $R(x) = P(x)Q(x)$ has the two positive roots $\beta$ and $\beta^{-1}$, so $Q(x)$ has no positive roots.
\end{proof}

\subsection{Salem numbers with small expansions}

The purpose of this section is to identify certain families of degree 6 Salem numbers which have eventually periodic expansions. 
First, we introduce some new notation.
\begin{definition}
Let
\begin{enumerate}[(1)]
\item $T_n$ denote the set of all Salem numbers of degree $n$;
\item $P_{m,p}$ denote the set of Salem Parry numbers whose beta expansions have preperiod $m$ and period $p$;
\item $C_{Q(x)}$ denote the set of all Salem numbers whose co-factor is $Q(x)$.
\end{enumerate}
\end{definition}

The following is a known result from Boyd, translated into this notation.
\begin{proposition}[\cite{boyd-s6}*{Proposition 4.1}]\label{prop4.1boyd}
We have $\beta \in T_6 \cap P_{1, 5}$ if and only if the minimal polynomial $P(x) = (a, b, c)$ of $\beta$ satisfies $a \leq b \leq 0$ and $a \leq c \leq 0$ and $a \leq -1$. Moreover, $T_6 \cap P_{1, 5} = P_{1, 5} \cap C_{Q_1(x)}$ where $Q_1(x)$ is the constant function $Q_1(x) = 1$.
\end{proposition}

This result gives a complete description of degree 6 Salem polynomials which admit a finite beta expansion with $(m, p) = (1, 5)$. It also says that the only valid co-factor for such an expansion is $Q(x) = 1$. This is quite reasonable; if $\beta \in T_6 \cap P_{1, 5}$, then the companion polynomial $R(x)$ has degree $1+5 = 6$, so $R(x) = P(x)$ factors trivially.

Boyd gave more results of this nature concerning $T_6 \cap P_{1, 6}$ and $T_6 \cap P_{5, 33}$. We extend these results in Proposition \ref{x^2+1} and Proposition \ref{x^2-x+1} in order to gain a better understanding of $T_6 \cap P_{1,7}$.

We first recall a useful result. Given any sequence of digits, the following gives a simple criterion in order to detect whether the sequence is a beta expansion of some Parry number.
\begin{lemma}[Parry's Criterion \cite{parry}*{p.407}]
Let $(c_n)_{n \geq 1}$ be a sequence in $\N^{\N^+}$. Then $(c_n)_{n \geq 1}$ is the beta expansion for some $\beta > 1$ if and only if $(c_1, c_2, \dots) >_\text{lex} (c_k, c_{k+1}, \dots)$ for all $k \geq 2$ (where $>_\text{lex}$ denotes the lexicographic ordering on sequences).
\end{lemma}

One may obtain similar results to Proposition \ref{prop4.1boyd} for arbitrary $T_n \cap P_{m, p}$ in the following way. First, choose a fixed co-factor $Q(x)$ of the appropriate degree, and then compute $R(x) = P(x)Q(x)$. The coefficients of $R(x)$ define a sequence of length $(m, p)$ according to \eqref{eq:rxm1}. One may then determine, via Parry's criterion, the conditions on the coefficients of $P(x)$ under which this sequence is a valid beta expansion.

The following two results demonstrate this technique for $T_6 \cap P_{1, 7}$.
We automate this process in Section \ref{sec:computational results}.

\begin{proposition}\label{x^2+1}
If $\beta$ is a Salem number with minimal polynomial $P(x) = (a, b, c)$, then $\beta \in P_{1,7} \cap C_{x^2 + 1}$ if and only if
\begin{enumerate}[(i)]
\itemsep-0.4em
\item $b \leq -1$
\item $1 \leq c \leq -a$
\item $a \leq 2b$
\item If $a = 2b$ then $1 \leq b+c$.
\end{enumerate}
\end{proposition}
\begin{proof}
Let $\beta$ be the Salem root of $P(x) = (a, b, c)$, with $a, b,c$ satisfying the conditions (i) to (iv). Consider the co-factor $Q(x) = x^2 + 1$.
We have
\begin{equation}\label{eq:companion}
P(x)(x^2 + 1) = x^8 + ax^7 + (b+1)x^6 + (a+c)x^5 + 2bx^4 + (a+c)x^3 + (b+1)x^2 + ax + 1.
\end{equation}
Now, comparing coefficients with \eqref{eq:rxm1}, this defines the expansion given by
\begin{equation}\label{eq:expansion}
c_1\colon c_2, \dots, c_8 = -a:-b-1,-a-c,-2b,-a-c,-b-1,-a-1,-a-1.
\end{equation}
We claim that this is the beta expansion for some number $\beta' > 0$; that is, that the sequence $(c_n)_{n \geq 1}$ lexicographically dominates the sequence $(c_n)_{n \geq k}$ for all $k \geq 1$. By combining conditions (i) and (iii) we have $c_1 > c_2 = c_6$ and with (ii) we have $c_1 > c_3 = c_5$. Condition (iii) shows only the weak inequality $c_1 \geq c_4$, but if $c_1 = c_4$ then with (iv) we have
\begin{equation*}
 c_2 = -b-1 \geq -b-(b+c) = -2b-c = -a-c = c_5.
\end{equation*}
If we happen to have $c_1 = c_4$ \emph{and} $c_2 = c_5$, then $c_3 = c_6$ but $c_4 > c_7$, showing that indeed $(c_n)_{n \geq 1}$ dominates $(c_n)_{n \geq 4}$ after all. Finally, it is clear that $c_1 > c_7$ and $c_1 > c_8$. 
Therefore \eqref{eq:expansion} indeed satisfies Parry's criterion.

Suppose, as a result, that $\beta'$ is a Parry number with expansion \eqref{eq:expansion}. Then $\beta'$ has companion polynomial $R(x) = P(x)(x^2 + 1)$. Since $\beta'$ is clearly not a root of $x^2 + 1$, it must be a root of $P(x)$. Hence $\beta' = \beta$, and so $\beta$ has the expansion given. We note that $p = 7$ and not another divisor of $7$ (namely, 1), since if $p = 1$ then in particular we would have $c_2 = c_4$, but this is forbidden by (i).

Conversely, if $\beta \in P_{1,7} \cap C_{x^2 + 1}$, then the companion polynomial of $\beta$ is
\[ R(x) = x^8 - c_1x^6 - \dots - c_6x^2 - (c_7 + 1)x - (c_8 - c_1) = P(x)(x^2 + 1). \]
Comparing coefficients with \eqref{eq:companion}, we see that $c_1 = -a$, $c_2 = -b-1$, and so on, so that $\beta$ has the expansion given in \eqref{eq:expansion}. It is easy to see that $P(x)$ must satisfy the conditions (i) to (iv) by noticing that by relaxing any of the conditions we allow an invalid sequence $(c_n)_{n \geq 1}$. For example, if we relax (i) to allow $b = 0$ then $c_2 = -1$, which is impossible. If we relax (ii) to allow $c = 0$ then $c_1 = c_3$ but $-b-1 = c_2 < c_4 = -2b$, and so the sequence $(c_n)_{n \geq 1}$ does not dominate $(c_n)_{n \geq 3}$. The remaining cases are similar and are left for the reader.
\end{proof}

\begin{proposition}\label{x^2-x+1}
If $\beta$ is a degree six Salem number with minimal polynomial $P(x)$, then $\beta \in P_{1,7} \cap C_{x^2 - x + 1}$ if and only if
\begin{align}
P(x) &= x^6 - x^4 - x^3 - x^2 + 1, \text{or} \label{eq:p1} \\
P(x) &= x^6 - x^4 - 2x^3 - x^2 + 1, \text{or} \label{eq:p2} \\
P(x) &= x^6 - 2x^4 - 3x^3 - 2x^2 + 1.\label{eq:p3}
\end{align}
\end{proposition}

\begin{proof}
It is readily verified (for example, by explicitly computing $\beta$-expansions) that the $\beta$'s with minimal polynomials \eqref{eq:p1} to \eqref{eq:p3} each have $\beta$-expansions with $(m, p) = (1, 7)$ and co-factor $x^2 - x + 1$.

Conversely, suppose that $\beta \in P_{1,7} \cap C_{x^2 - x + 1}$. We will show that $\beta$ has minimal polynomial given by either \eqref{eq:p1}, \eqref{eq:p2}, or \eqref{eq:p3}. Suppose $\beta$ has minimal polynomial $P(x) = (a, b, c)$. Then its companion polynomial is $R(x) = P(x)(x^2 - x + 1)$.
Comparing coefficients with \eqref{eq:rxm1}, we see that $-c_1$ is the $x^7$ coefficient of $P(x)(x^2 - x + 1)$; that is, $c_1 = -a+1$. Similarly, $c_2 = c_6 = a-b-1$, $c_3 = c_5 = -a+b-c$, and $c_4 = -2b+c$.

By Parry's criterion, $c_1 \geq c_3$ and $c_1 \geq c_4$. We claim that either $c_1 = c_3$ or $c_1 = c_4$. Suppose otherwise; so suppose that $c_1 > c_3$ and $c_1 > c_4$, i.e. $c \geq b$ and $2b \geq a+c$. Then
\[ 2b \geq a+c \geq a+b \implies 0 \geq a-b, \]
but this is inconsistent with the fact that $c_2 \geq 0$.

Now if $c_1 = c_3$, it is easily seen that $a \geq 0$ and so $a = 0$. A similar argument shows that if $c_1 = c_4$ then $a = 0$. Hence in all cases, $a = 0$. However, the only Salem polynomials with $a = 0$ are \eqref{eq:p1}, \eqref{eq:p2}, or \eqref{eq:p3} as well as $x^6 - 4x^4 -7x^3 - 4x^2 + 1$. A simple calculation reveals that $(m, p) \neq (1, 7)$ for this latter case, and so we are done.
\end{proof}

\begin{remark}
Combining Propositions \ref{x^2+1} and \ref{x^2-x+1}, we have that
\begin{equation}\label{eq:supseteq}
T_6 \cap P_{1, 7} \supseteq P_{1, 7} \cap (C_{x^2 + 1} \cup C_{x^2 - x + 1}).
\end{equation}
Do we have equality? A quick computation confirms that the answer is no; the Salem polynomial $P(x) = (-5, 6, -7)$ has $(m, p) = (1, 7)$, and yet its co-factor is neither $x^2 + 1$ nor $x^2 - x + 1$ (it is $x^2 + 2x + 1$). We would ideally like to extend \eqref{eq:supseteq} to be an equality, but, in order to do so, we would seemingly need a complete list of all of the co-factors of expansions with $(m, p) = (1, 7)$.
\end{remark}

Using $x^2 + 2x + 1$ as a co-factor, we arrive at
\begin{proposition}\label{x^2+2x+1}
If $\beta$ is a Salem number with minimal polynomial $P(x) = (a,b,c)$, then $\beta \in P_{1,7} \cap C_{x^2 + 2x + 1}$ if and only if
\begin{enumerate}[(i)]
\itemsep-0.4em
\item $a \leq -4$
\item $2 - 2a \leq 2b \leq -3a -2$
\item $a - 2b + 2 \leq 2c \leq-2a -4b$.
\end{enumerate}
\end{proposition}
\begin{proof}
Left to the reader. Refer to Proposition \ref{x^2+1} and Proposition \ref{x^2-x+1} for the techniques needed.
\end{proof}

We will show in Section \ref{sec:computational results} that in fact 
\[ T_6 \cap P_{1, 7} = P_{1, 7} \cap (C_{x^2 + 1} \cup C_{x^2 - x + 1} \cup C_{x^2 + 2x + 1}). \]
In other words, we have found all of the co-factors for expansions with $(m, p) = (1, 7)$.

\section{Computing co-factors}\label{sec:computing co-factors}

In this section, we introduce techniques for searching for co-factors. Remember that if $Q(x)$ is a co-factor for some beta expansion, then the roots of $Q(x)$ all lie in the disk $|z| \leq (1 + \sqrt{5})/2$. Bounds on the roots imply bounds on the coefficients, and hence for all $n$ and $(m, p)$, the set of all possible co-factors for elements in $T_n \cap P_{m, p}$ is finite. So, at the very least, there are only a finite number of co-factors to consider. In principle, the methods of the previous section may be used on each of the potential co-factors separately to obtain a complete description of $T_n \cap P_{m, p}$ in terms of coefficients of minimal polynomials.

\subsection{Bounding the number of co-factors}

The purpose of this section is to say more about the set of co-factors for a given $n$ and $(m, p)$ other than that it is finite.
We first formally define the notion of the co-factor set of $T_n \cap P_{m, p}$.
\begin{definition}
Suppose that
\begin{equation}\label{eq:T_ncapP_{m,p}}
T_n \cap P_{m, p} = P_{m, p} \cap \left( \bigcup_{p(x) \in \Lambda(n, m, p)} C_{p(x)} \right)
\end{equation}
for some $\Lambda(n, m, p) \subseteq \Z[x]$, and that $\Lambda(n, m, p)$ is the minimal set (with respect to set inclusion) which satisfies \eqref{eq:T_ncapP_{m,p}}. Then $\Lambda(n, m, p)$ is called the \emph{minimal co-factor set} for $T_n \cap P_{m, p}$. When $n$, $m$, and $p$ are understood, we may write $\Lambda$ instead of $\Lambda(n, m, p)$.
\end{definition}

\begin{example}
The minimal co-factor set for $T_6 \cap P_{1, 5}$ is $\Lambda = \{ 1\}$. Refer to Proposition \ref{prop4.1boyd}.
\end{example}

\begin{proposition}
For all $n$, $m$, and $p$, $\Lambda(n, m, p)$ is finite.
\end{proposition}
\begin{proof}{}
By the above discussion.
\end{proof}

The bound $|z| \leq (1+\sqrt{5})/2$ may be used to construct explicit finite approximations for $\Lambda$. The following provides a demonstration for $T_6 \cap P_{1, 7}$.

\begin{lemma}\label{bndQcoeffs}
We have $\Lambda(6, 1, 7) \subseteq \Gamma$, where
\[ \Gamma = \{ x^2 + d_1 + d_0 \in \Z[x] :  |d_1| \leq 3 \text{ and } 0 \leq d_0 \leq 2\}. \]
\end{lemma}
\begin{proof}
Let $Q(x) = x^2 + d_1x + d_0 \in \Z[x]$ be a co-factor with roots $\alpha_1, \alpha_2$. Then
\[ Q(x) = x^2 + d_1x + d_0 = x^2 - (\alpha_1 + \alpha_2)x + \alpha_1\alpha_2, \]
so
\[ |d_0| = |\alpha_1||\alpha_2| \leq \left(\frac{1+\sqrt{5}}{2}\right)^2 \approx 2.6 \]
and
\[ |d_1| = |\alpha_1 + \alpha_2| \leq |\alpha_1| + |\alpha_2| \leq 1 + \sqrt{5} \approx 3.2. \]
Moreover, note that if $(c_n)_{n \geq 1}$ is the beta expansion of $\beta$, then from \eqref{eq:rxm1} we have $c_1 = -a - d_1$ and $c_8 = -a -d_1 - d_0$. Parry's criterion guarantees $c_1 \geq c_8$, so $d_0 \geq 0$.
\end{proof}

\begin{remark}
The set $\Gamma$ described in Lemma \ref{bndQcoeffs} is far from minimal. To emphasize this point, note that this result gives $|\Gamma| = 21$ potential co-factors for elements of $T_6 \cap P_{1,7}$; however, of these, 8 have roots outside the disk $|z| \leq (\sqrt{5} + 1)/2$, immediately disqualifying them from being co-factors. 
As it turns out (see Proposition \ref{P_{1,7}}), the minimal co-factor set of $T_6 \cap P_{1, 7}$ has cardinality 3.
\end{remark}

Our goal is to find methods to bound $|\Lambda(n, m, p)|$ that are stronger than the one offered by simply considering the bound on the roots of the co-factors.
The following result is a straightforward generalization of \cite{boyd-s4}*{Lemma 1}, but will nonetheless prove to be useful in this regard.
\begin{lemma}\label{bnda}
If $\beta$ is a Salem Parry number with minimal polynomial $P(x) = (a, b, c)$, then $-a-5 \leq \lfloor \beta \rfloor \leq -a + 3$.
\end{lemma}
\begin{proof}
Let $\beta$, $\beta^{-1}$, $\alpha_1$, $\alpha_1^{-1}$, $\alpha_2$, $\alpha_2^{-1}$ be the roots of $P(x)$. Then
\[ -a = \beta + \beta^{-1} + 2\text{Re}(\alpha_1) + 2\text{Re}(\alpha_2) \]
hence
\[ -a - \beta^{-1} - 4 < \beta < -a - \beta^{-1} + 4 \]
so $-a-5 \leq \lfloor \beta \rfloor \leq -a + 3$.
\end{proof}

Using Lemma \ref{bnda}, a constant bound can be found on the $x^{n-1}$ coefficient of all co-factors.
\begin{proposition}\label{bnd d_n-1}
If $\beta$ is a Salem Parry number with co-factor $Q(x) = x^n + d_{n-1}x^{n-1} + \dots + d_0$, then $-4 \leq d_{n-1} \leq 5$.
\end{proposition}
\begin{proof}
Suppose $\beta$ has minimal polynomial $P(x) = (a, b, c)$ with eventually periodic expansion $(c_n)_{n \geq 1}$. First note that in the algorithm for calculating greedy expansions, we have $c_1 = \lfloor \beta \rfloor$, so, by Lemma \ref{bnda}, we have $-a-5 \leq c_1 \leq -a+3$.

Now, 
\[ R(x) = P(x)Q(x) = x^{n+6} + (d_{n-1} + a)x^{n+5} + \cdots. \]
We compare the coefficients of $R(x)$ with \eqref{eq:r(x)}. If $p \geq 2$, then $-c_1 = d_{n-1} + a$, which implies $-3 \leq d_{n-1} \leq 5$. If $p = 1$, then $-c_1 = d_{n-1} + a + 1$ so $-4 \leq d_{n-1} \leq 4$. In either case, the claimed bound holds.
\end{proof}

\begin{remark}
The techniques used to prove Proposition \ref{bnd d_n-1} can be used to bound $d_{n-2}$ in terms of the coefficients of $P(x)$. Unfortunately, no constant bound seems to exist for $d_{n-2}$.
It is certainly the case that no constant bound can exist on all of the coefficients of $Q(x)$ at once, since the examples exhibited in Section \ref{sec:arbitrarily large periods} have co-factors with arbitrarily large coefficients.
\end{remark}

The results presented so far in this section can be summarized by the following theorem.
\begin{proposition}\label{bnd lambda}
We have
\[ |\Lambda(n, m, p)| \leq 10 \left(\left\lfloor\left( \frac{1 + \sqrt{5}}{2} \right)^{m+p-n}\right \rfloor + 1 \right )\prod_{k = 2}^{m+p-n-1} B(n, m, p, k). \]
where
\[ B(n, m, p, k) = 2 \left \lfloor\binom{m+p-n}{k}\left(\frac{\sqrt{5} + 1}{2}\right)^k \right \rfloor + 1 \]
\end{proposition}
\begin{proof}
Let $\ell = m+p-n$ and suppose that $Q(x) = x^{\ell} + d_{\ell-1}x^{\ell-1} + \dots + d_0$ is the co-factor of some $\beta \in T_n \cap P_{m, p}$. If $\alpha_1, \dots, \alpha_\ell$ are the roots of $Q(x)$, then
\begin{align*}
|d_{\ell-1}| &= |\alpha_1 + \dots + \alpha_\ell| \leq \ell \left(\frac{\sqrt{5} + 1}{2}\right) \\
|d_{\ell-2}| &= |\alpha_1\alpha_2 + \dots + \alpha_1\alpha_\ell + \dots + \alpha_{\ell-1}\alpha_\ell| \leq \binom{\ell}{2} \left(\frac{\sqrt{5} + 1}{2}\right)^2
\end{align*}
which, in general, becomes
\begin{align*}
|d_{\ell-k}| \leq \binom{\ell}{k} \left(\frac{\sqrt{5} + 1}{2}\right)^k.
\end{align*}
So there are at most $2 \left\lfloor \binom{\ell}{k} \left(\frac{\sqrt{5} + 1}{2}\right)^k \right\rfloor + 1$ choices for $d_{n-k}$. In other words,
\begin{align*}
|\Lambda(n, m, p)| &\leq \prod_{k = 1}^{m+p-n} \left(2\left\lfloor\binom{m+p-n}{k}\left(\frac{\sqrt{5} + 1}{2}\right)^k \right \rfloor + 1\right) \\
&= \prod_{k = 1}^{m+p-n} B(n, m, p, k)
\end{align*}
By Proposition \ref{bnd d_n-1}, there are only $10$ choices $d_{n-1}$, and by the same argument in the proof of Lemma \ref{bndQcoeffs}, we must have $d_0 \geq 0$. The result follows.
\end{proof}

\subsection{Computational results}\label{sec:computational results}

In this section we compute explicit minimal co-factor sets for $T_6 \cap P_{1, p}$ for small values of $p$.
The algorithm is a generalization of the methods used to prove Propositions \ref{x^2+1} to \ref{x^2+2x+1}.
First, we find an upper bound $\Gamma$ for $\Lambda(6, 1, p)$ via the method of Lemma \ref{bndQcoeffs}. For each $Q(x) \in \Gamma$, if $Q(x)$ has roots outside of $|z| \leq (1 + \sqrt{5})/2$ then we can immediately reject it as a potential co-factor.
This can be quickly done via the Routh-Hurwitz method (see Section \ref{sec:alternative methods of bounding}).
Otherwise, $Q(x)$ defines a unique beta expansion $(c_n)_{n \geq 1}$ of length $(m, p)$, in terms of the coefficients of $P(x)$, which can be obtained by computing $P(x)Q(x)$ and comparing the coefficients with \eqref{eq:r(x)}. Writing $c_k = c_k(a, b,c)$, we obtain the following system of linear inequalities in $a, b$, and $c$:
\begin{equation}\label{eq:geq0}
\begin{rcases*}
\phantom{c_{m+p}}c_1(a, b, c) \geq 0 \\
\phantom{c_{m+p}}c_2(a, b, c) \geq 0 \\
\phantom{c_{m+p}c_2(a, b, c)}\vdots \\
\phantom{c_1}c_{m+p}(a, b, c) \geq 0
\end{rcases*}
\end{equation}
and, from Parry's criterion,
\begin{equation}\label{eq:geq_parry}
\begin{rcases*}
\phantom{c_{m+p}}c_1(a, b, c) \geq c_2(a, b, c) \\
\phantom{c_{m+p}}c_1(a, b, c) \geq c_3(a, b, c) \\
\phantom{c_{m+p}c_2(a, b, c)}\vdots \\
\phantom{c_{1}}c_{1}(a, b, c) \geq c_{m+p}(a, b, c)
\end{rcases*}
\end{equation}
If the set of degree 6 Salem polynomials which satisfy these systems is empty, then we can safely conclude that $Q(x)$ cannot be a cofactor and so $\Lambda(6, m, p) \subseteq \Gamma \setminus \{Q(x)\}$. We can trim $\Gamma$ in this fashion until no more co-factors can be removed, at which point we have found a closer approximation to the minimal co-factor set for $T_6 \cap P_{m, p}$. 
This process can be automated by using an integer linear programming framework.

Note that the systems \eqref{eq:geq0} and \eqref{eq:geq_parry} are necessary but not sufficient conditions for $Q(x)$ to be a co-factor. There are several complications can that arise. For instance, \eqref{eq:geq_parry} is only a partial implementation of Parry's criterion; in a full implementation, one would need to verify, for example, that if $c_1(a, b, c) = c_2(a, b, c)$ then $c_2(a, b, c) \geq c_3(a, b, c)$, and so on.
Moreover, if a polynomial $P(x) = (a, b,c)$ satisfies \eqref{eq:geq0} and \eqref{eq:geq_parry}, there is no guarantee that $P(x)$ is Salem, nor that it has expansion exactly $(m, p)$ (it could have period a divisor of $p$). One instructive example is for $T_6 \cap P_{1, 10}$. By choosing co-factor $Q(x) = x^5 + 1$, we obtain candidate beta expansion
\begin{align*}
c_1(a,b,c) &= -a && c_7(a,b,c) = -b \\
c_2(a,b,c) &= -b && c_8(a,b,c) = -c \\
c_3(a,b,c) &= -c && c_9(a,b,c) = -b \\
c_4(a,b,c) &= -b && c_{10}(a,b,c) = -a-1 \\
c_5(a,b,c) &= -a-1 && c_{11}(a,b,c) = -a-1 \\
c_6(a,b,c) &= -a-1
\end{align*}
The systems \eqref{eq:geq0} and \eqref{eq:geq_parry} are consistent with respect to this sequence, and one example solution is $(a, b, c) = (-3, 0, 0)$. This sequence also satisfies Parry's criterion and $P(x) = (-3, 0, 0)$ \emph{is} a Salem polynomial, so presumably $x^5 + 1 \in \Lambda(6, 1, 10)$. However, this is not true. This expansion has $p = 5$, not $10$.

Let $\Gamma'$ be the subset of $\Gamma$ obtained after the process of trimming $\Gamma$ using \eqref{eq:geq0} and \eqref{eq:geq_parry}. We have just demonstrated that $\Lambda(6, m, p) \neq \Gamma'$ in general. However, usually $\Gamma'$ will be relatively small, and can simply be checked by hand. Typically for each $Q(x) \in \Gamma'$ one will either find a Salem number which has $Q(x)$ as a co-factor, showing indeed that $Q(x) \in \Lambda(6, m, p)$, or if no example seems to exist, then $Q(x)$ is likely in one of the extraneous cases described in the preceding paragraph.

We now apply this technique to explicitly find $\Lambda(6, 1, 7)$ and $\Lambda(6, 1, 8)$.
\begin{proposition}\label{P_{1,7}}
The minimal co-factor set of $T_6 \cap P_{1, 7}$ is  
\[ \Lambda = \{x^2 - x + 1, x^2 + 1, x^2 + 2x + 1\}. \] 
\end{proposition}
\begin{proof}
In Lemma \ref{bndQcoeffs}, we showed that an upper bound for $\Lambda(6, 1, 7)$ is
\[ \Gamma = \{ x^2 + d_1 + d_0 \in \Z[x] :  |d_1| \leq 3 \text{ and } 0 \leq d_0 \leq 2\}. \]
A quick computation determines that the only elements of $\Gamma$ for which the systems \eqref{eq:geq0} and \eqref{eq:geq_parry} are simultaneously consistent are $x^2 - x + 1$, $x^2 + 1$, and $x^2 + 2x + 1$. From Propositions \ref{x^2+1}, \ref{x^2-x+1}, and \ref{x^2+2x+1}, we see that these three polynomials are indeed valid co-factors. 
\end{proof}
\begin{remark}
This confirms that Propositions \ref{x^2+1}, \ref{x^2-x+1}, and \ref{x^2+2x+1} give a complete description of degree 6 Salem numbers whose beta expansions have $(m, p) = (1, 7)$.
\end{remark}

\begin{proposition}
The minimal co-factor set of $T_6 \cap P_{1, 8}$ is  
\[ \Lambda = \{x^3 + x^2 + x + 1, x^3 + 2x^2 + 2x + 1\}. \]
\end{proposition}
\begin{proof}
An upper bound for $\Lambda(6, 1, 8)$ is
\[ \Gamma = \{ x^3 + d_2x^2 + d_1x + d_0 \in \Z[x] : |d_2| \leq 4 \text{ and } |d_1| \leq 7 \text{ and } 0 \leq d_0 \leq 4\}. \]
There are seven elements of $\Gamma$ such that \eqref{eq:geq0} and \eqref{eq:geq_parry} are simultaneously consistent. They are
\begin{align}
x^3 - x^2 - x\label{eq:poly1} \\
x^3 - x^2 - x + 1\label{eq:poly2} \\
x^3 - x^2 + 1 \label{eq:poly3}\\
x^3 - x \label{eq:poly4}\\
x^3 +1 \label{eq:poly5}\\
x^3+x^2 +x +1 \label{eq:poly6}\\
x^3+2x^2 +2x +1\label{eq:poly7}.
\end{align}
Note that \eqref{eq:poly3} and \eqref{eq:poly5} do not result in sequences which satisfy Parry's criterion. Moreover, \eqref{eq:poly1}, \eqref{eq:poly2}, and \eqref{eq:poly4} force $a = 0$, but there are only four degree 6 Salem numbers with $a = 0$ and none of them have $(m, p) = (1, 8)$. The only remaining options are \eqref{eq:poly6} and \eqref{eq:poly7}. These are indeed valid co-factors: for example, the Salem polynomial $P(x) = (-2, 0, 1)$ has \eqref{eq:poly6} as a co-factor and $P(x) = (-4, 6, -7)$ has \eqref{eq:poly7} as a co-factor.
\end{proof}

\subsection{Alternative methods of bounding $\Lambda(n, m, p)$}\label{sec:alternative methods of bounding}

The algorithm from the previous section can be used to determine $\Lambda(6,1 , p)$ for small values of $p$. The execution speed depends mainly on the size of the upper bound $\Gamma$, and typically $\Gamma$ tends to be much larger than $\Lambda$. Since checking roots is typically much faster than verifying feasibility of \eqref{eq:geq0} and \eqref{eq:geq_parry}, in practice one usually computes the roots of $Q(x) \in \Gamma$ to first confirm that its roots lie in $|z| \leq (1 + \sqrt{5})/2$.
As a result, a lot of time is wasted on computing the roots of potential co-factors. 
The purpose of this section is to develop methods to filter out these polynomials without having to compute their roots.

To this end, we introduce the Routh-Hurwitz criterion, a root test which is commonly used in control system theory.
Suppose that $p(x) = a_nx^n + a_{n-1}x^{n-1} + \dots + a_0 \in \R[x]$, with $a_n \neq 0$. We can construct the so-called \emph{Routh table} of $p(x)$ as follows:
\begin{center}
\begin{tabular}{ c|cccc } 
 $x^n$ &  $a_n$ & $a_{n - 2}$ & $a_{n - 4}$ & $\cdots$ \\
 $x^{n-1}$ & $a_{n-1}$ & $a_{n-3}$ & $a_{n-5}$ & $\cdots$ \\
 $x^{n-2}$ & $b_1$ & $b_2$ & $b_3$ & $\cdots$ \\ 
 $x^{n-3}$ & $c_1$ & $c_2$ & $c_3$ & $\cdots$ \\ 
 $\vdots$ & $\vdots$ & $\vdots$ & $\vdots$ & $\ddots$ \\ 
 $x^0$ & $k_1$ & $k_2$ & $k_3$ & $\cdots$
\end{tabular}
\end{center}
where
\begin{align*} 
b_1 &= -\frac{1}{a_{n-1}}\begin{vmatrix} a_n & a_{n-2} \\ a_{n-1} & a_{n-3} \end{vmatrix} \\[1em]
b_2 &= -\frac{1}{a_{n-1}}\begin{vmatrix} a_n & a_{n-4} \\ a_{n-1} & a_{n-5} \end{vmatrix} \\[1em]
c_1 &= -\frac{1}{b_1}\begin{vmatrix} a_{n-1} & a_{n-3} \\ b_1 & b_2 \end{vmatrix}
\end{align*} 
and so on. Note that in constructing the Routh table, we divide by elements of the first column; hence if any of the elements in the first column are zero then the Routh table may not be defined.

\begin{theorem}[Routh-Hurwitz Criterion \cite{dorf}*{pp. 391--399}]
The polynomial $p(x)$ has all of its roots in the open left half plane if and only if its Routh table is defined and all of the first-column elements have the same sign.
\end{theorem}

\begin{example}
The Routh table of $p(x) = x^4 + 2x^3 + 7x^2 + 4x + 3$ is
\begin{center}
\begin{tabular}{ c|cccc } 
 $x^4$ &   $1$ & $7$ & $3$ \\
 $x^3$ &  $2$ & $4$ & $0$ \\
 $x^2$ &  $5$ & $3$ & $0$ \\
 $x^1$ &  $14/5$ & $0$ & $0$ \\
 $x^0$ &  $3$ & $0$ & $0$ \\
\end{tabular}
\end{center}
Since all of the elements in the first column are positive, $p(x)$ has all of its roots in the open left half plane. On the other hand, the polynomial $p(x) = x^4 + 7x^2 + 4x + 3$ has a root outside the open left half plane since its Routh table is not defined.
\end{example}

The Routh-Hurwitz criterion is useful to us because, with a few small modifications, it allows us to answer questions about roots in open disks instead of the open left half plane. 
Say $Q$ is a polynomial, and $f$ is a linear fractional transformation which maps the open left half plane to the disk $|z| < \alpha$ in the complex plane. Then $z_0$ is a root of $Q \circ f$ if and only if $f(z_0)$ is a root of $Q$. It follows that all of the roots of $Q \circ f$ are in the open left half plane if and only if all of the roots of $Q$ are in the disk $|z| < \alpha$. 

Here is perhaps a more instructive example. Let $Q(x) = x^2 + ax + b \in \Z[x]$. The linear fractional transformation
\[ f(z) = 2 \cdot \frac{1+z}{1-z} \]
maps the open left hand plane to the disk $|z| < 2$. Now, consider
\[ Q(f(z)) = \left(2 \cdot \frac{1+z}{1-z}\right)^2 + a\left(2 \cdot \frac{1+z}{1-z}\right) + \left (2 \cdot \frac{1+z}{1-z} \right). \]
 In order to answer the question of whether or not $Q(x)$ has roots inside the disk $|z| < 2$, we may apply the Routh-Hurwitz criterion to $Q(f(z))$. This is because any root of $Q(f(z))$ in the open left hand plane will be a root of of $Q(x)$ in the image of the open left hand plane under $f$, namely, $|z| < 2$. Of course, $Q(f(x))$ is not a polynomial, so we cannot apply Routh-Hurwitz directly to it. However, its roots are the same as the roots of its numerator, which \emph{is} a polynomial.

The following result provides an example of how this can be used to search for co-factors.
\begin{proposition}
\label{filter}
Let $Q(x) = x^3 + d_2x^2 + d_1x + d_0 \in \Z[x]$. Define
\begin{align*}
a_3 &= d_0 + 2d_1 + 4d_2 + 8 \\
a_2 &= -3d_0 - 2d_1 + 4d_2 + 24 \\
a_1 &= 3d_0 - 2d_1 - 4d_2 + 24 \\
a_0 &= -d_0 + 2d_1 - 4d_2 + 8.
\end{align*}
If $a_3 \neq 0$, and either $a_0a_3 \leq 0$ or $a_2a_3 \leq 0$ or $a_1a_2 \leq a_0a_3$, then $Q(x)$ is not the co-factor of any Salem number.
\end{proposition}
\begin{proof}
The proof is by contrapositive. Suppose $Q(x)$ is the co-factor for some Salem Parry number $\beta$. Then $Q(x)$ must have all of its roots inside the closed disk $|z| \leq (1 + \sqrt{5})/2$. In particular, all of its roots are in the open disk $|z| < 2$. Define the linear fractional transformation $f \colon \C \to \C$ by
\[ f(z) = 2\cdot\frac{1+z}{1-z}. \]
Note that $f$ maps the open left half plane to the open disk $|z| < 2$. We have
\begin{align*}
Q(f(z)) &= \left(2\frac{1+z}{1-z}\right)^3 + d_2\left(2\frac{1+z}{1-z}\right)^2 + d_1\left(2\frac{1+z}{1-z}\right) + d_0 \\
&= \frac{(d_0 + 2d_1 + 4d_2 + 8)z^3}{(z - 1)^3} + \frac{(-3d_0 - 2d_1 + 4d_2 +24)z^2}{(z - 1)^3} \\
&\quad+ \frac{(3d_0 - 2d_1 - 4d_2 + 24)z}{(z - 1)^3} + \frac{-d_0 + 2d_1 - 4d_2 + 8}{(z - 1)^3}.
\end{align*}
Let $a_0, a_1, a_2$, and  $a_3$ be as in the statement of the proposition so that the roots of $Q \circ f$ are the roots of $h(x) := Q(f(z))(z-1)^3 = a_3x^3 + a_2x^2 + a_1x + a_0$. Since $a_3 \neq 0$, $h(x)$ is a degree 3 polynomial with Routh table
\begin{center}
\begin{tabular}{ c|cccc } 
 $x^3$ &  $a_3$ & $a_1$ \\
 $x^{2}$ & $a_2$ & $a_0$ \\
 $x^{1}$ & $b_1$ & $0$ \\
 $x^{0}$ & $a_0$ & $0$ \\
\end{tabular}
\end{center}
where
\[ b_1 = \frac{a_1a_2 - a_0a_3}{a_2}. \]
Since the roots of $Q(x)$ are all inside the disk $|z| < 2$, the roots of $Q \circ f$ are all in the open left half plane. Therefore there must be no sign changes in the first column of the Routh table, hence $a_2a_3 > 0$ and $a_0a_3 > 0$. Also, we must have $b_1/a_3 > 0$ and so 
\[ \frac{a_1a_2 - a_0a_3}{a_2a_3} > 0. \]
Since $a_2a_3 > 0$, this is equivalent to $a_1a_2 > a_0a_3$.
\end{proof}

The above result is of computational importance; it allows us to quickly filter down the list of possible co-factors. If we naively compute the upper bound $\Gamma$ for $\Lambda(6, 1, 10)$ as before, we would have
\[ |\Gamma| = 10 \cdot 53 \cdot 85 \cdot 69 \cdot 12 = 37301400 \]
co-factors to check. Of these, only 5609 have roots exclusively inside $|z| \leq (1 + \sqrt{5})/2$, immediately disqualifying the remaining $37295791$ from being co-factors. The proportion of polynomials in $\Gamma$ with roots outside $|z| \leq (1 + \sqrt{5})/2$ only grows as the degrees of the co-factors grows, and explicitly computing roots of potential co-factors to ensure that they all lie in $|z| \leq (1 + \sqrt{5})/2$ quickly becomes computationally infeasible. Generalizing the criteria of Proposition \ref{filter} to higher degree co-factors allows us to quickly filter down the list of possible co-factors without having to compute any roots. In the $\Lambda(6, 1, 10)$  case, the criteria reduces the list of $37301400$ potential co-factors down to just $158674$. This method can easily be generalized to Salem numbers of any degree.

Note that we may just as well have replaced $f \colon \C \to \C$ in the proof of Proposition \ref{filter} by
\[ f(z) = k \cdot \frac{1 + z}{1 - z} \]
for any $(1 + \sqrt{5})/2 < k < 2$ in order to produce stronger restrictions on $a_3, a_2, a_1$, and $a_0$. However, we usually choose $k = 2$ in practice because it provides adequate restriction while not causing floating-point complications. To avoid floating-point errors, one may also choose any rational approximation of $(1 + \sqrt{5})/2$ (from above) which is of the form $k/2^n$ for some $k, n \in \Z^+$, as numbers of this form have finite binary representations.

Table \ref{tab:Lambda(6, 1, p)} gives the explicit descriptions of $\Lambda(6, 1, p)$ for $5 \leq p \leq 10$.

\begin{table}
\centering
\begin{tabular}{ |r|l| }
 \hline
 $p$ & $\Lambda(6, 1, p)$ \\
 \hline
$5$ & $\phantom{s.}1$ \\
 \hline
$6$ & $\phantom{s.}x+1$ \\
 \hline
$7$ & \begin{tabular}{l}$x^2 + 1$ \\ $x^2 - x + 1$ \\ $x^2 + 2x + 1$\end{tabular} \\
 \hline
$8$ & \begin{tabular}{l}$x^3 + x^2 + x + 1$ \\ $x^3 + 2x^2 + 2x + 1$\end{tabular} \\
\hline
$9$ & \begin{tabular}{l}$x^4 - x^3+ x^2 - x + 1$ \\ $x^4 + x^3 + 2x^2 + x + 1$ \\ $x^4 + 3x^3 + 4x^2 + 3x + 1$ \end{tabular} \\
\hline
$10$ & \begin{tabular}{l}$x^5 + x^4 - x^3 - x^2 + x + 1$ \\ $x^5 + 2x^4 + 2x^3 + 2x^2 + 2x + 1$ \\ $x^5 + 3x^4 + 5x^3 + 5x^2 + 3x + 1$ \end{tabular} \\
 \hline
\end{tabular}
\caption{Complete descriptions of $\Lambda(6, 1, p)$ for small values of $p$}\label{tab:Lambda(6, 1, p)}
\end{table}

\section{Cyclotomic co-factors}\label{sec:cyclotomic co-factors}

As shown in \cite{boyd-s4}, degree 4 Salem polynomials always have cyclotomic co-factors. However, this is not true for higher degrees Salem polynomials, which becomes evident from Table \ref{tab:Lambda(6, 1, p)}.
The main result of this section, Proposition \ref{no cyclotomic}, gives a criterion on the coefficients of $P(x)$ which guarantee that no cyclotomic co-factor can exist. 

We first begin by proving a simple result which we will need for Proposition \ref{no cyclotomic}.

\begin{lemma}\label{l:pgeq3}
If $P(x)$ is a Salem polynomial with a reciprocal co-factor, then $p \geq 3$.
\end{lemma}
\begin{proof}
Suppose that $P(x)$ has reciprocal co-factor $Q(x)$ and that $p = 1$, with $\beta$-expansion given by
\[ c_1, c_2, \dots, c_m: c_{m+1}. \]
Note that the companion polynomial $R(x) = P(x)Q(x)$ is also reciprocal.
But, by definition,
\begin{align*}
R(x) &= P_{m+1}(x) - P_m(x) \\
&= x^{m+1} - (c_1 + 1)x^m - (c_2 - c_1)x^{m - 1} - \dots - c_{m+1} + c_m.
\end{align*}
Since $R(\beta) = 0$, we must have $\deg R(x) \geq 6$, so $m \geq 5$.
Comparing the $x^m$ and $x$ coefficients of $R(x)$, we see that $c_1 = c_m - c_{m - 1} - 1$ and hence that $c_1 < c_m$, contradicting Parry's criterion. Therefore $p \neq 1$.

Now suppose that $p = 2$, with $\beta$-expansion
\[ c_1, c_2, \dots, c_m: c_{m+1}, c_{m+2} \]
Again, $R(x)$ is reciprocal and is given by
\begin{align*}
R(x) &= P_{m+2}(x) - P_m(x) \\
&= x^{m+2} - c_1x^{m+1} - (c_2 + 1)x^m - \dots - c_{m+2} + c_m
\end{align*}
with $m \geq 4$. Comparing coefficients, we arrive at the identities
\begin{align}
c_m &= c_{m+2} +1  &&\text{comparing $x^{m+2}$ and $x^0$} \label{eq:x0}\\
c_1 &= c_{m+1} - c_{m - 1} &&\text{comparing $x^{m+1}$ and $x$} \label{eq:x1}\\
c_2 &= c_{m} - c_{m - 2} - 1 &&\text{comparing $x^{m}$ and $x^2$} \label{eq:x2}
\end{align}
Since $c_1 \geq c_{m+1}$, by Parry's criterion, \eqref{eq:x1} implies that $c_{m - 1} = 0$ and hence that $c_1 = c_{m+1}$. Substituting \eqref{eq:x0} into \eqref{eq:x2} gives
\[ c_2 = c_{m+2} - c_{m-2} \]
so by the same reasoning $c_2 = c_{m+2}$. Now, by Parry's criterion,
\begin{align*} 
(c_1, c_2, c_3, \dots) > (c_{m+1}, c_{m+2}, \dots) = (c_1, c_2, c_1, c_2, \dots).
\end{align*}
In particular, $c_3 \geq c_1$. However, by considering
\[ (c_1, c_2, c_3, \dots) > (c_3, c_4, c_5 \dots) \]
we also have $c_1 \geq c_3$, so $c_1 = c_3$. We can continue in this way to see that $c_k = c_{k+2}$ for all $k \geq 1$, contradicting the fact that $m > 1$.
\end{proof}

\begin{proposition}\label{no cyclotomic}
Suppose that $\beta$ is a Salem number with minimal polynomial $P(x) = (a, b, c)$. If $b \leq 2a -3$ then $P(x)$ has no cyclotomic co-factor.
\end{proposition}
\begin{proof}
Let $\Phi_n(x) = x^{\varphi(n)} + d_{1}x^{\varphi(n)-1} + d_2x^{\varphi(n)-2} + \dots + 1$ be the $n$th cyclotomic polynomial. Then $d_1 = -\mu(n)$ and
\[ d_2 = \begin{cases} \frac{1}{2}(\mu(n)^2 - \mu(n)) &\text{if $n$ is odd} \\ \frac{1}{2}(\mu(n)^2 - \mu(n)) - \mu(n/2) &\text{if $n$ is even}, \end{cases} \]
where $\mu(n)$ is the M\"{o}bius function, which is defined by
\[ \mu(n) = \begin{cases} 1 &\text{if $n$ is square-free with an even number of prime factors} \\ -1 &\text{if $n$ is square-free with an odd number of prime factors} \\ 0 &\text{if $n$ is not square-free}. \end{cases} \]
If $n$ is odd, then $d_2$ is either $0$ or $1$. We have $d_2 = 0$ when $d_1 = -\mu(n) \in \{-1, 0\}$ and $d_2  = 1$ when $d_1 = -\mu(n) = 1$. Analyzing the even case is similar (noting that if $\mu(n) = \pm 1$ then $\mu(n/2) = \mp 1$).
Therefore we can limit $d_1$ and $d_2$ to seven possibilities:
\[ (d_1, d_2) \in \{(-1, 0), (-1, 1), (0, -1), (0, 0), (0, 1), (1, 0), (1, -1) \}. \]
Suppose that $P(x)$ has a cyclotomic co-factor, so its $\beta$ expansion is periodic with expansion
\[ c_1, c_2, \dots, c_m : c_{m+1}, \dots, c_{m+p}. \]
We can write its companion polynomial as
\begin{align*}
R(x) = P(x)\Phi(x) = x^{n + 6} + (a + d_1)x^{n + 5} + (b + d_1a + d_2)x^{n+4} + \cdots.
\end{align*}
By Lemma \ref{l:pgeq3} , $p \geq 3$ and so $c_1 = -(a+d_1)$ and $c_2 = -(b+d_1a+d_2)$. By Parry's criterion,
\begin{align}
 -(a + d_1) = c_1 \geq c_2 = -(b + d_1a + d_2).\label{eq:1}
\end{align}
We can break this inequality into seven cases based on the possible values of $d_1$ and $d_2$, showing in each case that the condition $b \leq 2a - 3$ disallows them. For example, if $d_1 = 1$ and $d_2 = 0$ then \eqref{eq:1} gives $b \geq 1$, but $b \leq 2a-3$ implies $b \leq 2\cdot0 - 3 = -3$. The remaining cases are similar and are omitted.
\end{proof}

\begin{remark}
Many of the polynomials with large expansions listed in Table \ref{tab:largeexps} and Table \ref{tab:all data} satisfy the condition $b \leq 2a - 3$.
It would be interesting to see if more can be said about the co-factors of expansions which have small $b$.
\end{remark}

\section{Examples of expansions with large trace and large period}\label{sec:expansions for large trace}

In this section we investigate the properties of certain families of degree 6 Salem numbers which have large trace.
In particular, we prove that degree 6 Salem numbers can have arbitrarily large periods by exhibiting an infinite family of Salem polynomials, parameterized by $k \in \N$, which have $(m, p) = (1, 8k+6)$ for $k \geq 2$.

\subsection{An example with arbitrarily large trace}

We begin with a remark about the size of $C(\beta)$ when $\text{trace}(\beta)$ can be large. Recall that the probabilistic argument from \cite{boyd-s6} concluded that the size of $C(\beta)$ may be closely related to the length of the beta expansion $m + p$. More precisely, under this model we have
\begin{align*}
E(m+p)\begin{cases} < \infty &\text{if $C(\beta) < 1$}, \\ = \infty &\text{if $C(\beta) \geq 1$}. \end{cases}
\end{align*}
Although Boyd noted certain exceptions to this prediction, such as the Salem polynomial $P(x) = (-9, -37, -55)$ which has $C(\beta) = 6.6956$ but $m+p = 531230$, the data collected seems to confirm a direct relationship between $C(\beta)$ and $m+p$ for degree 6 Salem numbers.
However, this prediction seems to break down quite dramatically in certain cases. Consider the degree 6 polynomial $P(x) = (a, a+1, -2)$, where $a \leq -2$ is even. It has associated polynomial $U(x) = x^3 + ax^2 + (a - 2)x - (2 + 2a)$, which easily satisfies the conditions of Lemma \ref{ureqs}, noting that $U(x)$ is irreducible by Eisenstein's criterion with $p = 2$. Therefore $P(x)$ is a Salem polynomial. By \cite{boyd-s6}*{Prop. 4.1}, we see that the beta expansion for $P(x)$ has $m+p = 6$. However, this family of Salem polynomials achieves arbitrarily large $C(\beta)$. Recall that
\[ C(\beta) = \left(\frac{\pi}{6}\right)^2 \frac{\beta^5}{\text{disc}(\beta)^{1/2}}. \]
For $P(x) = (a, a+1, -2)$, we may calculate $\text{disc}(\beta)$ using a computer algebra system to see that it is a polynomial in $a$ with leading coefficient $-128a^9$. From Proposition \ref{bnda}, $\beta^5$ is approximately $-a^5$ for small $a$, so indeed $C(\beta) \to \infty$ as $a \to -\infty$.
This clearly suggests that there are other factors involved which are not taken into account by Boyd's probabilistic argument.

\subsection{An example with arbitrarily large period}\label{sec:arbitrarily large periods}

Although it is theorized that degree 6 Salem expansions can achieve arbitrarily large expansions, to date the largest expansion known is 
\[ (m, p) = (26440386599, 6051049471) \] 
which is achieved by the Salem polynomial $P(x) = (-5, -2, -11)$. Here, we exhibit a family of polynomials which achieve arbitrarily large periods.

\begin{lemma}\label{pissalem}
If $a \leq -6$ is divisible by $3$, then $P(x) = (a, -2a, 2a-3)$ is a Salem polynomial.
\end{lemma}
\begin{proof}
Consider $U(x) = x^3 + ax^2 - (2a+3)x - 3$ and apply Lemma \ref{ureqs}. We have $U(2) = -1 < 0$ and $U(-2) = 8a - 5 < 0$. Moreover, $U(x)$ is irreducible by Eisenstein's criterion with $p = 3$, so $U(n) \neq 0$ for any integer. Finally, $U(1) = -a-5 > 0$.
\end{proof}

The following result shows that there are degree 6 Salem numbers with arbitrarily long (but finite) periods.

\begin{proposition}\label{arbitrarily large}
If $k \geq 2$, then the degree 6 polynomial $P(x) = (a, -2a, 2a-3)$, where $a = -6k-3$, is Salem with eventually periodic expansion $(m, p) = (1, 8k+6)$.
\end{proposition}
\begin{proof}
By Lemma \ref{pissalem}, $P(x)$ is a Salem polynomial. Now consider the expansion
\begin{align}\label{eq:longexp}
6k&:6k-2, 6k, 2, \omega_1, 6(k-1), 6(k-1), 6(k-1), 6k-2, \nonumber\\
&0,2,1,1,2,0, 6k-2, 6(k-1), 6(k-1), 6(k-1), \omega_2, 2, \\
&6k, 6k-2, 6k-1, 6k-1 \nonumber
\end{align}
where $\omega_1$ and $\omega_2$ are sequences given by
\begin{align*} 
\omega_1 = 6, 6, 6, 9, \dots, 6(k-2), 6(k-2), 6(k-2), 6(k-2)+3 \\
\omega_2 = 6(k-2)+3, 6(k-2), 6(k-2), 6(k-2), \dots, 9, 6, 6, 6
\end{align*}
(so $\omega_2$ is the reverse of $\omega_1$). Note that the length of $\omega_1$ and $\omega_2$ is $4(k-2)$, so the sequence \eqref{eq:longexp} has total length $8k+7$.

For $k \geq 1$, it is easy to see that this sequence satisfies Parry's criterion. We have $c_1 > c_\ell$ for all $\ell$ except when $\ell = 3$ or $\ell = 8k+4$. When $\ell = 3$, we have $c_1 = c_3$ but $c_2 > c_4$, and when $\ell = 8k+4$ we have $c_1 = c_{8k+4}$ and $c_2 = c_{8k+5}$, but $c_3 > c_{8k+6}$. Moreover, we indeed have $(m, p) = (1, 8k+6)$ for this sequence. Note that $p$ must be $8k+6$, and not a (proper) divisor of $8k+6$, since the only occurrences of $6k$ in the period are $c_3$ and $c_{8k+4}$.

Now, consider the polynomial
\[ x^{8k+1} + 3x^{8k} + 5x^{8k-1} + 6x^{8k-2} + 7x^{8k-3} + \dots + 6kx^{4k+2} + 6kx^{4k+1}. \]
The coefficients are exactly the numbers less than $6k$ which are congruent to $1, 3, 5$, or $0$ modulo 6, written in increasing order. We may extend the coefficients of this polynomial to produce a reciprocal polynomial $Q(x)$. We claim that this $Q(x)$ is the co-factor for the expansion given in \eqref{eq:longexp}. This can be proven by explicitly computing the product $P(x)Q(x)$ and verifying that its coefficients agree with the expansion given in \eqref{eq:longexp}. We verify the sequence for $\omega_1$, that is, for the digits $c_5$ to $c_{4k-4}$, and leave the other cases for the reader.

Let $q_i$ be the coefficient of $x^i$ in $Q(x)$. Since the mod 6 residue classes of the coefficients of $Q(x)$ repeat with a period of 4, we may write
\begin{align}\label{eq:coeffsq}
q_i = q_{4m+r} =
\begin{cases}
6m+1, &r = 0 \\
6m+3, &r=1 \\
6m+5, &r=2 \\
6m+6, &r=3
\end{cases}
\end{align}
for $i < 4k$. Then, the $i$th coefficient of $P(x)Q(x)$ is
\begin{align*}
r_i &:= q_i - (6k+3)q_{i-1} + (12k+6)q_{i-2} - (12k+9)q_{i-3} \\
 &\quad+ (12k+6)q_{i-4} - (6k+3)q_{i-5} + q_{i-6}.
\end{align*}
If $4k > i \geq 6$, we may replace the coefficients $q_i$ to $q_{i-6}$ by reading directly from \eqref{eq:coeffsq}.
Again, we end up with different cases depending on the residue class of $i$ modulo 4. For example, if $i$ is a multiple of $4$, say $i = 4m$ for some $m \in \N$, we would have
\begin{align*}
r_{4m} &= 6m+1 - (6k+3)(6m) + (12k+6)(6m-1) - (12k+9)(6m-3) \\
&\quad+(12k+6)(6m-5) - (6k+3)(6m-6) + 6m-7 \\
&= -6m+3.
\end{align*}
The remaining cases can be computed in the same way. In general, we have
\begin{align}\label{eq:coeffsr}
r_i = r_{4m+r} =
\begin{cases}
-6m+3,   &r = 0 \\
-6m, &r=1 \\
-6m, &r=2 \\
-6m, &r=3
\end{cases}
\end{align}
as long as $4k > i \geq 6$. Note that since $P(x)Q(x)$ is reciprocal, we have $r_i = r_{m+p-i}$.

Since $m = 1$ for this expansion, we have $c_i = -r_{m+p-i} = -r_{i}$ for $i \leq m+p-2$ by \eqref{eq:rxm1}. Hence
\begin{align*}
c_i = c_{4m+r} = -r_{4m+r} = 
\begin{cases}
6m-3,   &r = 0 \\
6m, &r=1 \\
6m, &r=2 \\
6m, &r=3
\end{cases}
\end{align*}
as long as $4k > i \geq 6$ (remembering the condition on \eqref{eq:coeffsr}). One can check that this matches $\omega_1$ in this range. Of course, this range misses $c_5$, so we compute $c_5$ explicitly. The $m+p-5$ coefficient of $P(x)Q(x)$ is given by
\begin{align*} 
r_{m+p-5} &= -(6k+3)q_{m+p-6} + (12k+6)q_{m+p-7} - (12k+9)q_{m+p-8} \\
 &\quad+ (12k+6)q_{m+p-9} - (6k+3)q_{m+p-10} + q_{m+p-11} \\
 &= -(6k+3) + 3(12k+6) - 5(12k+9) + 6(12k+6) - 7(6k+3) + 9 \\
 &= -6
 \end{align*}
 so $c_5 = 6$, as required.
\end{proof}

\begin{remark}
Proposition \ref{arbitrarily large} also demonstrates the fact that the coefficients of co-factors can also be arbitrarily large. Moreover, while this result focuses on polynomials of the form $P(x) = (a, -2a, 2a-3)$ for $a = -6k-3$, computational evidence seems to suggest that the same result may be true if we replace $a$ by $-6k-4$ or $-6k-5$ for $k \geq 2$.
\end{remark}

\section{Computing beta expansions}

The computation of beta expansions is based on the greedy algorithm outlined in Section \ref{sec:introduction}. Having computed $c_1, \dots, c_{n}$, consider the polynomial
\[ P_n(x) = x^n - c_1x^{n-1} - \dots - c_{n} \]
so that $r_{n} = P_{n}(\beta)$.
If $\beta$ is a Salem number with minimal polynomial $P(x)$ then $B_n(\beta) = P_n(\beta)$ as long as $B_n(x) \equiv P_n(x) \mod P(x)$. Hence, the algorithm can be implemented by reducing $P_n(x)$ by $P(x)$. More precisely, set $B_0(x) := 1$ and $B_n(x) := x B_{n-1}(x) - c_n$ (so that $c_n = \lfloor \beta B_{n-1}(\beta) \rfloor$ and $r_n = B_n(\beta)$), where we reduce $B_n(x)$ modulo $P(x)$ at each step. This modular reduction keeps the degree of $B_n(x)$ at most the degree of $P(x)$, helping to minimize floating point error.

Since the coefficients of $B_n(x)$ can grow to be arbitrarily large, complications may arise when attempting to compute $c_n = \lfloor \beta B_{n-1}(\beta) \rfloor$ if $\beta B_{n-1}(\beta)$ is exceptionally close to an integer. 
Suppose we choose $\beta_0$ to be an approximation of $\beta$ to some fixed accuracy $\epsilon$. If $B_{n-1}(x) = b_dx^d + \dots + d_1x$, Boyd \cite{boyd-s6} demonstrates that $|\beta r_{n-1} - \beta_0 B_{n-1}(\beta_0)| < \eta$ where
\[ \eta = \sum |b_i|i(\beta_0 + \epsilon)^{i-1}\epsilon. \]
Hence if the distance from $\beta B_{n-1}(\beta)$ to the nearest integer is at most $\eta$, we have $\lfloor \beta_0 B_{n-1}(\beta_0) \rfloor = \lfloor \beta r_{n-1} \rfloor = c_n$ as desired. In this project, choosing $\epsilon = 5 \times 10^{-64}$ was sufficient for all computations.

\subsection{Finding $(m, p)$}\label{sec:finding (m,p)}

Determining the explicit expansion size $(m, p)$ for large expansions is typically done in three steps and is based on the algorithm from \cite{boyd-s6} with minor modifications. The first step is detecting periodicity. We compute $B_n(x)$ and keep track of the $n$ for which the trailing coefficient of $B_n(x)$ is the largest in absolute value found so far; that is, $n$ for which $|L(B_n)| > |L(B_k)|$ for all $k < n$ where $L(B_n)$ denotes the trailing coefficient of $B_n(x)$. If a new record is found at $N_0$, say, then we can safely conclude that $m+p > N_0$. This is the method through which the upper bounds on $m+p$ are computed in Table \ref{tab:largeexps} and Table \ref{tab:all data}.
For periodic expansions, we know that record values must stop occurring after a certain point. Hence if we have computed $B_n(x)$ up to $N$ and the last record was at $N_0$, with $N \gg N_0$, then this may suggest that the sequence is periodic. 

Once an expansion is suspected to be periodic, we may explicitly attempt to find its period. This may be done in a straightforward way if we store a table of all $B_n$ up to $N$; however, for large $N$, this becomes infeasible. Instead, we do the following. We store values of $B_n$ for all $n$ divisible by $10^7$, and when we want to search for a period, we re-start the computation at the largest multiple of $10^7$ less than $N$, say $M$. The idea is that $M$ is likely already in the periodic region of the expansion rather than the preperiod. At each step, we compute $B_{M + n}$ and if at any point we detect that $B_{M} = B_{M + n}$, then the period must be $n$.

Once $p$ is found, there are several ways to find the preperiod $m$. The simplest method is to restart the computation from the beginning, and at each step compute $B_n$ and $B_{n+p}$. The preperiod is then the smallest $n$ for which $B_n = B_{n+p}$. This method is typically only feasible for small values of $m$. A faster way is to take advantage of the stored values of $B_n$ for $n$ divisible by $10^7$. For any multiple $M$ of $10^7$, we can determine whether or not $M < m$ by re-starting the computation of $B_n$ at $n = M$ and checking if $B_{M} = B_{M + p}$. Then, a binary search can be used to find $m$.

Boyd utilized several techniques which we did not implement here. For one, he outlines an algorithm based on integer arithmetic which outright avoids the issue of floating point errors.
Unfortunately, Boyd noted that this algorithm appears to be much slower than the one based on floating point arithmetic. 
Second, he suggests the use of what are called ``markers'' in order to quickly determine the period of an expansion by inspection. 
A marker is simply an ordered pair $(n, B_n)$ where $B_n$ satisfies some unusual property. 
For example, throughout the computation one may store all ordered pairs $(n, B_n)$ where the trailing coefficient of $B_n$ is divisible by $1000$.
By visually inspecting the list of these markers, a period can usually be detected by inspection unless the period is unusually small. 
Of course, the use of the trailing coefficient of $B_n$ here is arbitrary, and any other suitable marker can be used.

\section{Comments and open questions}

In Section \ref{sec:arbitrarily large periods} we showed that $p$ can be arbitrarily large for degree 6 Salem numbers. 
It would be interesting to investigate whether or not the same is true for $m$.
The largest value of $m$ found so far is $m = 26440386599$ which is achieved by the Salem polynomial $P(x) = (-5, -2, -11)$.
It may be possible to exhibit a family of degree 6 Salem numbers with arbitrarily large $m$ in much the same way as we have here for $p$.
While it appears that we have $m = 1$ for the large majority of degree 6 Salem numbers (at least for small trace), 
it seems that values of $m > 1$ occur with some regularity.

Another question that we can ask is, what more can we say about the co-factors of these expansions?
In Section \ref{sec:computing co-factors}, we demonstrated a constant bound on the $d_{n-1}$ coefficient. Perhaps other constant bounds exist on the other coefficients.
Even without constant bounds, it may be possible to find bounds that are stronger than the ones outlined in Section \ref{sec:computing co-factors} which would lend themselves well to the computation of these co-factors.
Moreover, it would seem that there is a much richer structure to these co-factors in the case when $m = 1$. For example, in Section \ref{sec:salem numbers with (m,p)=(1,p)} we showed that co-factors for such expansions have no positive roots.
Additionally, all co-factors with $m = 1$ which we have investigated have been reciprocal, so it may be worthwhile to investigate whether or not this is always the case.

\section*{Acknowledgments}

The author would like to thank their supervisor, Kevin Hare, for his guidance, 
stimulating discussions, 
and for his helpful feedback on the manuscript.

\bibliography{notes}

\newpage

{\small\tabcolsep=3pt
\begin{ThreePartTable}
\begin{TableNotes}
\footnotesize
\item [$*$] Many expansions are still unknown. For these, lower bounds on $m+p$ are given.
\item [$\dagger$] Values are truncated.
\item [$\ddagger$] $C(\beta) = (\pi/6)^2 \beta^5 / \text{disc}(\beta)^{1/2}$.
\item[$\mathsection$] Records indicate the largest value of $|L(B_n)|$ found. 
\end{TableNotes}
\begin{longtable}{ r|c|r|c|r|r }
\caption{All previously unknown beta expansions for $\text{trace}(\beta) \leq 15$}\label{tab:all data}
\endfirsthead

 \insertTableNotes
 \endlastfoot
 $(a, b, c)$ & $\beta^\dagger$ & $m+p >$ & $(m, p)$ & $C(\beta)^{\dagger\ddagger}$ & Record$^\mathsection$ \\
 \hline
 $(-3, -1, -7)$     &  $3.78$ &   $77930975072$ & $*$ &  $0.3342$ & $514138$ \\  
 $(-5, -2, -11)$    &  $5.70$ &     N/A & $(26440386599, 6051049471)$ &  $0.5350$ & N/A \\
 $(-6, -26, -39)$   &  $9.28$ &      $86487966351$ & $*$ &  $7.3275$ & $1787395$ \\
 $(-6, -21, -31)$   &  $8.81$ &     $84721475914$ & $*$ &  $1.4877$ & $1114807$ \\
 $(-7, -29, -43)$   & $10.26$ &   N/A & $(1039779, 90)$ & $1.6081$ & N/A \\
 $(-7, -28, -41)$   & $10.17$ &    N/A & $(159781799, 94829)$ &  $0.7047$ & N/A \\
 $(-8, -33, -49)$   & $11.32$ & $29951657970$ & $*$ &  $2.9741$ & $718862$ \\
 $(-8, -30, -44)$   & $11.08$ &      N/A & $(2121493281, 188611456)$ &  $1.5932$ & N/A \\
 $(-8, -26, -38)$   & $10.76$ &    $90758317726$ & $*$ &  $1.4038$ & $1004347$ \\
 $(-8, -23, -34)$   & $10.51$ &       $82585386156$ & $*$ & $2.6313$ & $1416014$ \\
 $(-8, -3, -17)$    &  $8.58$ &   $33565256553$ & $*$ &  $1.7055$ & $683933$ \\
 $(-9, -35, -51)$   & $12.22$ &       N/A & $(4075324464, 425719617)$ &  $0.8332$ & N/A \\
 $(-9, -28, -41)$   & $11.70$ &    $77441722314$ & $*$ & $17.1476$ & $4483265$ \\
 $(-9, -6, -20)$    &  $9.82$ &    $41820094414$ &$*$ &  $1.1025$ & $704069$ \\
 $(-10, -41, -61)$  & $13.41$ &   $58743875586$ &$*$ & $40.9628$ & $5073073$ \\
 $(-10, -40, -59)$  & $13.34$ &      $76519283803$ &$*$ &  $2.3151$ & $1458931$ \\
 $(-10, -36, -52)$  & $13.07$ &   N/A & $(23087045143, 820400)$ &  $0.6838$ & N/A \\
 $(-10, -4, -21)$   & $10.57$ &     $33346343238$ &$*$ &  $0.9056$ & $496260$ \\
 $(-11, -41, -60)$  & $14.19$ & $4509221815$ &$*$ & $3.3354$ & $272459$ \\
 $(-11, -40, -58)$  & $14.13$ & $4975465702$ &$*$ & $1.0121$ & $155555$ \\
 $(-11, -39, -55)$  & $14.06$ & $5702355915$ &$*$ & $0.4030$ & $162379$ \\
 $(-11, -35, -49)$  & $13.80$ & N/A & $(48516722, 3128603)$ & $0.2436$ & N/A \\
 $(-11, -33, -48)$  & $13.68$ & $6520342537$ &$*$ & $3.2277$ & $343165$ \\
 $(-11, -30, -44)$  & $13.48$ & $3082910698$ &$*$ & $2.4631$ & $202174$ \\
 $(-11, -14, -28)$  & $12.32$ & N/A & $(1490333, 72458)$ & $1.211$ & N/A \\
 $(-11, -11, -26)$  & $12.09$ & N/A & $(1285570, 677)$   & $0.8665$ & N/A \\
 $(-12, -48, -71)$  & $15.42$ & $5332844313$ & $*$& $6.0252$ & $513512$ \\
 $(-12, -47, -69)$  & $15.36$ & N/A & $(18789419, 12521)$ & $2.0830$ & N/A \\
 $(-12, -44, -63)$  & $15.18$ & $5615819438$ & $*$& $0.6271$ & $277263$ \\
 $(-12, -43, -62)$  & $15.12$ & N/A & $(779478947, 96687)$ & $0.7936$ & N/A \\
 $(-12, -42, -61)$  & $15.06$ & $4794842329$ & $*$& $2.5353$ & $341849$ \\
 $(-12, -40, -56)$  & $14.94$ & $5393947499$ & $*$& $0.3002$ & $114214$ \\
 $(-12, -32, -47)$  & $14.45$ & $5725301822$ & $*$& $6.5609$ & $742843$ \\
 $(-12, -16, -31)$  & $13.37$ & $4449213713$ & $*$& $2.5857$ & $316914$ \\
 $(-12, -13, -29)$  & $13.16$ & $5258791867$ & $*$& $3.1256$ & $385702$ \\
 $(-12, -7, -26)$   & $12.71$ & $5891492389$ &$*$& $12.6842$ & $563167$ \\
 $(-12, 10, -25)$   & $11.30$ & $3772492995$ & $*$& $4.4985$ & $480889$ \\
 $(-13, -51, -75)$  & $16.40$ & $4684449005$ & $*$& $2.8476$ & $273591$ \\
 $(-13, -49, -71)$  & $16.28$ & $5385287106$ &$*$ & $1.0417$ & $164413$ \\
 $(-13, -48, -70)$  & $16.23$ & $5121034206$ &$*$ & $2.5854$ & $356368$ \\
 $(-13, -46, -67)$  & $16.12$ & $4557447235$ &$*$ & $22.5968$ & $1591305$ \\
 $(-13, -45, -65)$  & $16.06$ & $5364294220$ &$*$& $1.2787$ & $381811$ \\
 $(-13, -43, -62)$  & $15.95$ & $4664822736$ &$*$ & $1.3348$ & $290440$ \\
 $(-13, -41, -59)$  & $15.83$ & $5438222938$ &$*$ & $1.0866$ & $239409$ \\
 $(-13, -38, -55)$  & $15.66$ & N/A & $(191196227, 16067)$ & $1.9660$ & N/A \\
 $(-13, -35, -51)$  & $15.48$ & $4850105137$ & $*$ & $1.4524$ & $199990$ \\
 $(-13, -30, -41)$  & $15.16$ & N/A & $(137293807, 164656)$ & $0.0885$ & N/A \\
 $(-13, -18, -34)$  & $14.41$ & $4797790801$ &$*$ & $15.0895$ & $685912$ \\
 $(-13, -10, -29)$  & $13.87$ & $3097038217$ &$*$ & $1.6788$ & $197382$ \\
 $(-13, -6, -27)$   & $13.59$ & N/A & $(10913948, 13914931)$ & $0.4296$ & N/A \\
 $(-13, 33, -45)$   & $10.15$ & $5353738198$ &$*$ & $1.9069$ & $417931$ \\
 $(-14, -55, -81)$  & $17.43$ & $3988555193$ &$*$ & $3.9560$ & $330172$ \\
 $(-14, -54, -79)$  & $17.37$ & $5277973189$ &$*$ & $1.9842$ & $244874$ \\
  $(-14, -46, -66)$   & $16.95$ & $2482472077$ & $*$& $0.8903$ & $113161$ \\
 $(-14, -45, -65)$  & $16.89$ & $3077396744$ &$*$ & $3.7782$ & $403613$ \\
 $(-14, -43, -62)$  & $16.79$ & $2706985583$ &$*$ & $2.0378$ & $278308$ \\
 $(-14, -40, -58)$  & $16.62$ & $1871742376$ &$*$ & $5.4243$ & $270753$ \\
 $(-14, -38, -55)$  & $16.51$ & $2009498508$ &$*$ & $0.8809$ & $139034$ \\
 $(-14, -37, -54)$  & $16.45$ & $1746326378$ &$*$ & $2.1663$ & $205357$ \\
 $(-14, -36, -45)$  & $16.37$ & N/A & $(2098011, 112)$ & $0.0895$ & N/A \\
 $(-14, -30, -41)$  & $16.03$ & $483556715$ & $*$& $0.0764$ & $15520$ \\
 $(-14, -16, -34)$  & $15.20$ & $1059562394$ & $*$& $1.5574$ & $111094$ \\
 $(-14, 13, -29)$   & $13.17$ & N/A & $(1428555, 7640)$ & $0.3080$ & N/A \\
 $(-14, 41, -57)$   & $10.60$ & $931823664$ &$*$ & $12.3727$ & $213478$ \\
 $(-15, -58, -85)$  & $18.41$ & $1004214780$ &$*$ & $2.5336$ & $163315$ \\
 $(-15, -56, -82)$  & $18.31$ & $981795645$ &$*$ & $19.5120$ & $300321$ \\
 $(-15, -55, -80)$  & $18.26$ & $747403126$ &$*$ & $2.2502$ & $135618$ \\
 $(-15, -54, -77)$  & $18.20$ & $915780516$ &$*$ & $0.6918$ & $80500$ \\
 $(-15, -46, -66)$  & $17.80$ & $653186611$ &$*$ & $1.0825$ & $64206$ \\
 $(-15, -45, -65)$  & $17.74$ & $870686637$ & $*$& $8.4716$ & $278642$ \\
 $(-15, -43, -62)$  & $17.64$ & N/A & $(11994574, 217750)$ & $1.4952$ & N/A \\
 $(-15, -40, -58)$  & $17.48$ & $626960901$ &$*$ & $1.1215$ & $89321$ \\
 $(-15, -39, -57)$  & $17.43$ & $926257584$ &$*$ & $3.6077$ & $258594$ \\
 $(-15, -21, -39)$  & $16.42$ & $570142197$ &$*$ & $2.0194$ & $67444$ \\
 $(-15, -18, -37)$  & $16.25$ & $978181595$ &$*$ & $7.1219$ & $240549$ \\
 $(-15, -13, -33)$  & $15.94$ & $945829876$ &$*$ & $0.2394$ & $40150$ \\
 $(-15, -11, -33)$  & $15.82$ & $981760213$ &$*$ & $1.5239$ & $78853$ \\
 $(-15, -6, -31)$   & $15.51$ & $862707241$ &$*$ & $0.7906$ & $82583$ \\
 $(-15, -5, -31)$   & $15.45$ & $634681995$ &$*$ & $5.2527$ & $132679$ \\
 $(-15, 19, -34)$   & $13.79$ & $503771992$ &$*$ & $0.6744$ & $42752$ \\
 $(-15, 37, -51)$   & $12.31$ & $754961281$ &$*$ & $10.7717$ & $139783$
\end{longtable}
\end{ThreePartTable}
}

\end{document}